\documentclass[12pt,reqno]{amsart}

\usepackage{amsfonts,amsmath,amsthm}
\usepackage{amssymb,epsfig}
\usepackage{enumerate}
\usepackage{bbm}
\usepackage{fancyhdr}
\usepackage{cleveref}

\usepackage[text={425pt,650pt},centering]{geometry}
\usepackage{graphicx}
\usepackage{epsfig}
\usepackage{tikz}
\usepackage{caption}
\usepackage{color} 
\definecolor{vert}{rgb}{0,0.6,0}
\usepackage{pgfplots}
\pgfplotsset{compat=1.17}

\numberwithin{figure}{section}

\theoremstyle{plain}
\newtheorem{defn}{Definition}[section]
\newtheorem{thm}[defn]{Theorem}

\newtheorem{ex}[defn]{Example}

\newtheorem{cor}[defn]{Corollary}

\newtheorem*{rem}{\textbf{Remark}}
\numberwithin{equation}{section}

\newcommand{\R}{\mathbb{R}}

\newcommand{\cB}{\mathcal{B}}

\newcommand{\BUC}{{\rm BUC\,}}

\newcommand{\RNum}[1]{\uppercase\expandafter{\romannumeral #1\relax}}

\begin{document}

\title
{The vanishing viscosity process for an eikonal equation in the radially symmetric setting}

\pagestyle{fancy}
\fancyhf{}  
\fancyhead[CE]{\footnotesize F. Meng}                          
\fancyhead[CO]{\footnotesize Vanishing Viscosity for Eikonal Equation}
\fancyhead[LE]{\footnotesize \thepage}  
\fancyhead[RO]{\footnotesize \thepage}  
\renewcommand{\headrulewidth}{0pt}

\author
{Fanchen Meng}

\thanks{
}

\address
{
Department of Mathematics, 
Cornell University,  212 Garden Ave, Ithaca, NY 14853, USA}
\email{fm484@cornell.edu}


\maketitle
\begin{abstract}
    We study the vanishing viscosity method for the eikonal equation $|Du|=V$ in $\cB(0,1)$ with homogeneous Dirichlet boundary value condition. By assuming $V$ is radially symmetric and restricting attention to radially symmetric solutions, we construct explicit formulas for both the viscous solution $u^{\epsilon}$ and the limiting solution $u$. We prove $u^{\epsilon}\rightarrow u$ as $\epsilon \rightarrow 0^+$ qualitatively and quantitatively derive an $\epsilon |\log \epsilon|$ type local convergence rate. Finally, we discuss the uniqueness of viscosity solutions for the eikonal equation and give some examples. 
\end{abstract}
\section{Introduction}
In this paper, we consider the following eikonal equation with homogeneous Dirichlet boundary value condition
\begin{equation}\label{eikonal}
\left\{
\begin{aligned}
|Du| &= V \quad \text{in } \cB(0,1), \\
u &= 0 \quad \text{on } \partial \cB(0,1),
\end{aligned} 
\right.
\end{equation}
where $\cB(0,1) $ is the unit ball in $\mathbb{R}^n$, and $V: \cB(0,1) \rightarrow [0,\infty) $ is a nonnegative Lipschitz continuous function.

The eikonal equation arises as a fundamental model in the theory of first-order Hamilton-Jacobi equations. It represents the stationary Hamilton-Jacobi equation with a Hamiltonian of the form $H(p,x)=|p|-V(x)$ for $(p,x)\in \R^n\times \R^n$, and plays a central role in both the analysis of nonlinear PDEs and in applications such as geometric optics, front evolution, and optimal control. 

To study the solutions of \eqref{eikonal}, we use the vanishing viscosity procedure. For $\epsilon>0$, we consider the equation
\begin{equation}\label{vanishingviscosity}
\left\{
\begin{aligned}
    |Du^{\epsilon}|-V &= \epsilon \Delta u^\epsilon \quad \text{in } \cB(0,1), \\
    u^\epsilon &=0 \quad \text{on } \partial \cB(0,1).   
\end{aligned}
\right.
\end{equation}
To simplify our problem, we always assume that $V$ is radially symmetric and consider the radially symmetric solution of \eqref{eikonal}.
With a slight abuse of notation, we suppose $V(x)=V(|x|)=V(r)$, where $V: [0,1) \rightarrow [0,\infty)$ is Lipschitz continuous. 

The vanishing viscosity method is a classical regularization technique for studying first-order Hamilton-Jacobi equations (see \cite{CrandallLions1983, CrandallEvansLions1984, Tran2021} and the references therein). By adding a small diffusive term $\epsilon \Delta u^{\epsilon}$ to the original equation, the new equation admits smooth solutions for each $\epsilon>0$. Classically, this process selects a correct viscosity solution out of many possible almost everywhere solutions. If the limiting equation has the comparison principle, then it has a unique viscosity solution, and hence, the unique viscosity solution is selected. 

     It is important to emphasize that when studying the convergence of solutions to \eqref{vanishingviscosity} as $\epsilon \to 0$, a so-called  ``selection problem''arises. Specifically, for our problem, the limiting PDE \eqref{eikonal} admits multiple viscosity solutions as $V$ can be zero at some locations, which is discussed in detail in Section 5. It is not clear at all what solution that $\{u^{\epsilon} \}_{\epsilon>0}$ converges to, and what is the underlying selection principle. In general, for such a selection problem, there has been a lot of related research (see \cite{Bessi, Gomes2003, Evans, AIPS, Yu, Gomes2008, Cagnetti2011, ishii2017vanishing, ishii2017boundary, TuZhang} and the references therein), but this remains a widely open field. For further discussions and a broader view of related open problems, we refer the reader to \cite[Section 6.6]{LeMitakeTran2017}, where these issues are explored in more detail.

In this paper, we establish the following results in the radially symmetric setting of the eikonal equation:

\begin{enumerate}
    \item \textbf{Selection principle:} We prove that the maximal solution among all radially symmetric viscosity solutions is selected by the vanishing viscosity process (Theorem~3.1).
    
    \item \textbf{Quantitative convergence rate:} We derive an explicit estimate for the convergence as $\epsilon \to 0$:
    \[
    |u^\epsilon(x) - u(x)| \le C \epsilon |\log \epsilon|\,(1 - |x| - \log|x|),
    \quad x \in \mathcal{B}(0,1) \setminus \{0\},
    \]
    as stated in Theorem~4.1.
    To the best of our knowledge, the results of Theorems~3.1 and 4.1 are new in the literature.
    
    \item \textbf{Uniqueness discussion:} We analyze the uniqueness of viscosity solutions for the eikonal equation when $V$ vanishes at certain points, and provide illustrative examples in Section~5.
    The result in this part is well known; see \cite{LeMitakeTran2017} for example.
\end{enumerate}

\section{Preliminaries}
As $V$ is radially symmetric, it is natural to look for radially symmetric solutions to \eqref{vanishingviscosity} where $u^{\epsilon}(x)=u^{\epsilon}(|x|)$ for $x\in \overline{\cB(0,1)}$. Now \eqref{vanishingviscosity} has been simplified into an ODE:
\begin{equation}\label{ODEvanishingviscosity}
\left\{
\begin{aligned}
    |(u^{\epsilon})^{'}|-V &= \epsilon\left((u^{\epsilon})^{''}+\frac{n-1}{r}(u^{\epsilon})^{'}\right)  &\quad \text{in} \ (0,1),\\
    u^\epsilon (1) &=0 &r=1.
\end{aligned} 
\right.
\end{equation}
To solve \eqref{ODEvanishingviscosity}, we first assume $(u^\epsilon )^{'}(r)\le 0$ for $r\in(0,1)$. We will check our assumption after we get $u^\epsilon$.

Let $p^\epsilon (r)= (u^\epsilon )^{'} (r) $, then $p^\epsilon$ satisfies a nonhomogeneous first order ODE:
\begin{equation}\label{ODE_p_epsilon}
    (p^{\epsilon})^{'} + \left(\frac{n-1}{r}+\frac{1}{\epsilon}\right)p^{\epsilon} = -\frac{V}{\epsilon}.
\end{equation}
Using the integrating factor method, we get the solution $p^{\epsilon}(r)$ to \eqref{ODE_p_epsilon}:
\[
    p^{\epsilon}(r)= -\frac{1}{\epsilon}r^{-n+1}e^{-\frac{r}{\epsilon}}  \int_{0}^{r} s^{n-1}e^{\frac{s}{\epsilon}}V(s)\,ds + Cr^{-n+1}e^{-\frac{r}{\epsilon}}
\]
where $C$ is a constant to be chosen.

Since $u^\epsilon$ is radially symmetric, we require $p^\epsilon$ to satisfy 
$\lim_{r\rightarrow 0^+}p^{\epsilon}(r)=0$ so that $(u^{\epsilon}) ^{'}(0)=0$ and $u^{\epsilon}$ is smooth enough at $r=0$. 
Therefore $C=0$, and $p^\epsilon$ becomes:
\begin{equation*}
      (u^{\epsilon})^{'}(r)= p^{\epsilon}(r)= -\frac{1}{\epsilon}r^{-n+1}e^{-\frac{r}{\epsilon}}  \int_{0}^{r} s^{n-1}e^{\frac{s}{\epsilon}}V(s)\,ds.
\end{equation*}
Recalling our boundary condition $u^{\epsilon}(1)=0$, we now can get $u^{\epsilon}$: For $r\in[0,1]$, set 
\begin{equation*}
    u^{\epsilon}(r) = \int_{1}^{r} p^{\epsilon}(s) \,ds.  
\end{equation*}
However, according to the process above, $u^{\epsilon}$ we found satisfies \eqref{ODEvanishingviscosity} and hence   \eqref{vanishingviscosity} classically in the open interval $(0,1)$. We still need to check whether $u^{\epsilon}$ satisfies \eqref{vanishingviscosity} at 0.

To verify this, we compute the partial derivatives of $u^{\epsilon}$ at 0.
Let $\{e_1,e_2,\cdots,e_n\}$ be the canonical basis of $\mathbb{R}^n$.
Since  
\begin{equation*}
    u^{\epsilon}(0)= \int_{1}^{0}(u^{\epsilon})^{'}(s)\,ds,   
\end{equation*}
by L'Hopital's rule,
\begin{equation*}
\begin{aligned}
    u_{x_i}^{\epsilon}(0) 
    &= \lim_{h \rightarrow 0^+}\frac{u^{\epsilon}(he_i)-u^{\epsilon}(0)}{h} \\
    &= \lim_{h \rightarrow 0^+} \frac{\int_{0}^{h}(u^{\epsilon})^{'}(s)\, ds}{h}\\
    &=\lim_{h\rightarrow 0^+}(u^{\epsilon})^{'}(h)\\
    &=(u^{\epsilon})^{'}(0)\\
    &=0,
\end{aligned}
\end{equation*}
which implies that $u^{\epsilon}\in C^1(\cB(0,1))$. Next, by the L'Hopital rule again,

\begin{equation*}
\begin{aligned}
    u_{x_i x_i}^{\epsilon}(0) &= \lim_{h \rightarrow 0^+}\frac{u^{\epsilon}_{x_i}(he_i)-u^{\epsilon}_{x_i}(0)}{h} =  \lim_{h \rightarrow 0^+}\frac{p^{\epsilon}(h)}{h} \\
    &= \lim_{h \rightarrow 0^+}\frac{ -\frac{1}{\epsilon}h^{-n+1}e^{-\frac{h}{\epsilon}}  \int_{0}^{h} s^{n-1}e^{\frac{s}{\epsilon}}V(s)\,ds }{h}\\
    &= \lim_{h \rightarrow 0^+}-\frac{1}{\epsilon} e^{\frac{-h}{\epsilon}}\cdot \lim_{h \rightarrow 0^+}\frac{ \int_{0}^{h} s^{n-1}e^{\frac{s}{\epsilon}}V(s)\,ds}{h^n}\\
    &= -\frac{1}{\epsilon}\lim_{h \rightarrow 0^+}\frac{ \int_{0}^{h} s^{n-1}e^{\frac{s}{\epsilon}}V(s)\,ds}{h^n}\\
    &= -\frac{1}{\epsilon}\lim_{h \rightarrow 0^+}\frac{h^{n-1}e^{\frac{h}{\epsilon}} V(h)}{nh^{n-1}} \\
    &= -\frac{V(0)}{n\epsilon}. 
\end{aligned}
\end{equation*}
In this way,
\begin{equation*}
    Du^{\epsilon}(0) = (u^{\epsilon}_{x_1}(0),u^{\epsilon}_{x_2}(0),\cdots ,u^{\epsilon}_{x_n}(0))=(0,0,\cdots,0) ,
\end{equation*}
and
\begin{equation*}
     \epsilon\Delta u^{\epsilon}(0)= \epsilon \sum_{i=1}^n u^{\epsilon}_{x_i x_i}(0)= -V(0).
\end{equation*}
Therefore, 
\begin{equation*}
    |Du^{\epsilon}(0)|-V(0) =0-V(0)=-V(0)= \epsilon \Delta u^{\epsilon}(0).
\end{equation*}
Hence, we know $u^\epsilon$ solves \eqref{vanishingviscosity} classically at $x=0$. 
At the same time, since $V(s) \ge 0$ on $[0,1)$, then 
\begin{equation*}
    (u^{\epsilon})^{'}(r)= p^{\epsilon}(r)= -\frac{1}{\epsilon}r^{-n+1}e^{-\frac{r}{\epsilon}}  \int_{0}^{r} s^{n-1}e^{\frac{s}{\epsilon}}V(s)\,ds \le 0,
\end{equation*}
which confirms that our assumption is correct.

We thus have the formulas for $u^{\epsilon}$: For $|x|=r\in[0,1]$,

\[ 
     (u^{\epsilon})^{'}(r) = p^{\epsilon}(r)= -\frac{1}{\epsilon}r^{-n+1}e^{-\frac{r}{\epsilon}}  \int_{0}^{r} s^{n-1}e^{\frac{s}{\epsilon}}V(s)\,ds,
\]
\begin{equation}\label{solution_formula}
u^{\epsilon}(r) = \int_1^r p^{\epsilon}(s) \, ds = \int_r^1 \left(\frac{1}{\epsilon}\int_0^s \left(\frac{t}{s}\right)^{n-1} e^{\frac{t-s}{\epsilon}}V(t)\, dt\right)\, ds. 
\end{equation}
And we set, for $|x|=r\in[0,1]$,
\begin{equation}\label{u(r)}
     u(r) = \int_r^1 V(s)\, ds. 
\end{equation}
We remark that $u$ set in \eqref{u(r)} is the maximal viscosity solution to \eqref{eikonal} among all possible radially symmetric viscosity solutions to \eqref{eikonal}.
\section{Qualitative convergence argument of the vanishing viscosity process}
In this section, we will show that $u^{\epsilon}$ converges to $u$ as $\epsilon \rightarrow 0^{+}$ and discuss the convergence rate in the next section. We begin by stating the main result and then provide a proof.
\begin{thm}\label{thm:convergence}
As in \eqref{solution_formula}--\eqref{u(r)}, let $\{u^\epsilon\}_{\epsilon>0}$ and $u$ be solutions to \eqref{vanishingviscosity} and \eqref{eikonal} respectively under the radially symmetric setting. Uniformly for  $0\le r\le1$, we have
\[
\lim_{\epsilon \to 0^+} u^\epsilon(r) = u(r).
\]
\end{thm}
\begin{proof}
We first show $-p^{\epsilon}(s)\rightarrow V(s)$ as $\epsilon\rightarrow 0^+$ for $s>0$.\\
Recall our formula for $-p^{\epsilon}$:
\[ 
-p^{\epsilon}(s)= \frac{1}{\epsilon}\int_{0}^{s} \left(\frac{t}{s}\right)^{n-1}e^{\frac{t-s}{\epsilon}}V(t)\,dt.
\] 
Let $t=s+\epsilon y$, then $p^{\epsilon}$ becomes:
\begin{equation*}
\begin{aligned}
    p^{\epsilon}(s) &= \int_{-s/\epsilon}^0\left(\frac{s+\epsilon y}{s}\right)^{n-1}e^y V(s+\epsilon y)\, dy\\
                &= \int \left(\frac{s+\epsilon y}{s}\right)^{n-1}e^y V(s+\epsilon y) \mathbbm{1}_{(-s/\epsilon,0)}(y)\, dy.
\end{aligned}
\end{equation*}
Also, we can write $V$ as:
\[
    V(s) = \int e^y V(s) \mathbbm{1}_{(-\infty,0)} (y)\, dy.
\]
We see that the integrand 
\[
    \left(\frac{s+\epsilon y}{s}\right)^{n-1}e^y V(s+\epsilon y) \mathbbm{1}_{(-s/\epsilon,0)}(y) \rightarrow e^y V(s) \mathbbm{1}_{(-\infty,0)} (y) 
\]
for $s>0$, as $\epsilon \rightarrow 0^+$, since $V$ is continuous. 

We also see that
\[
    \left|\left(\frac{s+\epsilon y}{s}\right)^{n-1}e^y V(s+\epsilon y) \mathbbm{1}_{(-s/\epsilon,0)}(y)\right| \le Me^y \mathbbm{1}_{(-\infty,0)}(y)              
\]
for some $M>0$ and fixed $s>0$, as $V$ is bounded. 

Since $ Me^y \mathbbm{1}_{(-\infty,0)}(y)\in$ $L^1 (\mathbb{R})$, by the dominated convergence theorem,
\[
    \int \left(\frac{s+\epsilon y}{s}\right)^{n-1}e^y V(s+\epsilon y) \mathbbm{1}_{(-s/\epsilon,0)}(y)\, dy \rightarrow \int e^y V(s)\mathbbm{1}_{(-\infty,0)}(y)\, dy,
\]
i.e., we have $p^{\epsilon}(s)\rightarrow V(s)$ as $\epsilon \rightarrow 0^+$, for $s>0$.
Notice that 
\[
    |p^{\epsilon}(s)| = \left|\int_{-s/\epsilon}^0\left(\frac{s+\epsilon y}{s}\right)^{n-1}e^y V(s+\epsilon y)\, dy\right|\le M\int_{-s/\epsilon}^0 e^y \, dy \le M \int_{-\infty}^0 e^y dy=M
\]
since $V$ is bounded. 

Applying the dominated convergence theorem again, we get 
\[
    \int_r^1 -p^{\epsilon}(s)\, ds \rightarrow \int_r^1 V(s)
\]
as $\epsilon \rightarrow 0^+$, as desired.
Hence we proved $u^{\epsilon}(r) \rightarrow u(r)$ for $r\in [0,1]$, as $\epsilon \rightarrow 0^+$.

As $|p^{\epsilon}|\le M$, we have $\{u^{\epsilon}\}_{\epsilon \in (0,1)}$ is equi-continuous. By the Azel\`{a}-Ascoli theorem and the above, we imply that $u^{\epsilon}\rightarrow u$ uniformly on $[0,1]$.
\end{proof}

\section{rate of convergence of the vanishing viscosity process}
As in the previous section, we first present the main result, namely the convergence rate, and then provide the proof.
\begin{thm}\label{thm:rateconvergence}
    Let $\{ u^\epsilon \}_{\epsilon > 0}$ and $u$ be the radially symmetric solutions to  \eqref{vanishingviscosity} and \eqref{eikonal}, respectively, as defined in \eqref{solution_formula}--\eqref{u(r)}. Then, for all $0<r\le 1$, we have
\[
    |u^{\epsilon}(r)-u(r)|\le C\epsilon|\log \epsilon|(1-r-\log r) 
\]
for some positive $C>0$ independent of $\epsilon,r\in (0,1]$.
\end{thm}
 We note that a similar convergence rate $\epsilon |\log \epsilon|$ appears in the recent works \cite{QianSprekelerTranYu2024, ChaintronDaudin2025, CirantGoffi2025} for the vanishing viscosity process with uniformly convex Hamiltonians.
 Our result is different as it is for the eikonal equation with $H(x,p)=|p|-V(x)$, which is convex but not uniformly convex in $p$. This resemblance might hint at underlying structural connections worth investigating further. 
\begin{proof}
To find the rate of convergence for $u^{\epsilon}\rightarrow u$, we first turn to focus on the convergence rate of their derivatives. \\
In the following part, we study the rate of convergence for $-p^{\epsilon}(r)\rightarrow V(r)$.

From the discussion above, we have 
\begin{equation*}
    -p^{\epsilon}(r)= \frac{1}{\epsilon}\int_0^r\left(\frac{s}{r}\right)^{n-1}e^{\frac{s-r}{\epsilon}}V(s)\, ds.
\end{equation*}
Let $\frac{s-r}{\epsilon}=t$, then 
\begin{equation*}
    -p^{\epsilon}(r)=\int_{-\frac{r}{\epsilon}}^0(1+\frac{\epsilon}{r}t)^{n-1}e^tV(r+\epsilon t)\, dt.
\end{equation*}
Now we consider $|V(r)-(-p^{\epsilon}(r))|$. We have
\begin{equation*}
\begin{aligned}
    |V(r)-(-p^{\epsilon}(r))|
    &=\left|\int_{-\infty}^0e^tV(r)dt-\int_{-\frac{r}{\epsilon}}^0(1+\frac{\epsilon}{r}t)^{n-1}e^tV(r+\epsilon t)\, dt\right| \\
    &=\left|\int_{-\infty}^{-\frac{r}{\epsilon}}e^tV(r)\,dt+\int_{-\frac{r}{\epsilon}}^0 e^tV(r)dt -\int_{-\frac{r}{\epsilon}}^0(1+\frac{\epsilon}{r}t)^{n-1}e^tV(r+\epsilon t)\, dt\right| \\
    &= \left|\int_{-\infty}^{-\frac{r}{\epsilon}}e^tV(r)\,dt\right|+\int_{-\frac{r}{\epsilon}}^0 e^t|V(r)-(1+\frac{\epsilon}{r}t)^{n-1}V(r+\epsilon t)|\, dt\\
    &= \text{\RNum{1}} + \text{\RNum{2}}.
\end{aligned}
\end{equation*}
At this point, we bound \RNum{1} and \RNum{2} separately. We first control \RNum{1},
\begin{equation*}
\begin{aligned}
    \text{\RNum{1}} &= \left|\int_{-\infty}^{-\frac{r}{\epsilon}}e^tV(r)\,dt\right|\\
                    &= |V(r)e^{-\frac{r}{\epsilon}}|\\
                    &= |V(r)|e^{-\frac{r}{\epsilon}}\\
                    &\le Ce^{-\frac{r}{\epsilon}},
\end{aligned}
\end{equation*}
where we use $V(r)$ is bounded.

As for \RNum{2}, we separate it into integrals on two different intervals and bound them respectively. 
\begin{equation*}
    \text{\RNum{2}} = \left(\int_{-\frac{r}{\epsilon}}^{\log\epsilon}+\int_{\log{\epsilon}}^0\right)e^t|V(r)-(1+\frac{\epsilon}{r}t)^{n-1}V(r+\epsilon t)|\,dt
    = \text{\RNum{3}}+\text{\RNum{4}}.
\end{equation*}
We compute
\begin{equation*}
\begin{aligned}
    \text{\RNum{3}}&=\int_{-\frac{r}{\epsilon}}^{\log\epsilon}e^t|V(r)-(1+\frac{\epsilon}{r}t)^{n-1}V(r+\epsilon t)|\,dt\\
    &\le C\int_{-\frac{r}{\epsilon}}^{\log\epsilon}e^t\,dt\\
    &= \left.Ce^t \right|_{-\frac{r}{\epsilon}}^{\log{\epsilon}}\\
    &= C(\epsilon-e^{-\frac{r}{\epsilon}})\\
    &\le C\epsilon
\end{aligned}
\end{equation*}
for $r>0$.
Also, 
\begin{equation*}
\begin{aligned}
    \text{\RNum{4}}&= \int_{\log\epsilon}^{0}e^t|V(r)-(1+\frac{\epsilon}{r}t)^{n-1}V(r+\epsilon t)|\,dt\\
    &= \int_{\log\epsilon}^{0}e^t|V(r)-(1+\frac{\epsilon}{r}t)^{n-1}V(r)+(1+\frac{\epsilon}{r}t)^{n-1}V(r)-(1+\frac{\epsilon}{r}t)^{n-1}V(r+\epsilon t)|\,dt.
\end{aligned}
\end{equation*}
Now we focus on the integrand
\begin{equation*}
\begin{aligned}
    &|V(r)-(1+\frac{\epsilon}{r}t)^{n-1}V(r)+(1+\frac{\epsilon}{r}t)^{n-1}V(r)-(1+\frac{\epsilon}{r})^{n-1}V(r+\epsilon t)|\\
    &\le |V(r)|(1-(1+\frac{\epsilon}{r}t)^{n-1})+(1+\frac{\epsilon}{r}t)^{n-1}|V(r)-V(r+\epsilon t)|.
\end{aligned}
\end{equation*}
For the term $1-(1+\frac{\epsilon}{r}t)^{n-1}$, we see that 
\begin{equation*}
\begin{aligned}
    1-(1+\frac{\epsilon}{r}t)^{n-1}=1-(1+(n-1)\frac{\epsilon}{r}t+\cdots+(\frac{\epsilon}{r}t)^{n-1}).
\end{aligned}
\end{equation*}
Noticing that $-\frac{r}{\epsilon}\le \log{\epsilon}\le t\le 0$, we have $|\frac{\epsilon}{r}t|\le 1$.
Thus, 
\begin{equation*}    
|1-(1+\frac{\epsilon}{r}t)^{n-1}|\le C\frac{\epsilon}{r}|t|.
\end{equation*}
For the other term $(1+\frac{\epsilon}{r}t)^{n-1}|V(r)-V(r+\epsilon t)|$,
\begin{equation*}
\begin{aligned}
    &(1+\frac{\epsilon}{r}t)^{n-1}|V(r)-V(r+\epsilon t)|\\
    &\le |V(r)-V(r+\epsilon t)|\\
    &\le C\epsilon|t| 
\end{aligned}
\end{equation*}
where we use $V$ is Lipschitz continuous.

Finally, we add these bounds together to get
\begin{equation*}
\begin{aligned}
    \text{\RNum{4}}&= \int_{\log\epsilon}^{0}e^t|V(r)-(1+\frac{\epsilon}
    {r})^{n-1}V(r+\epsilon t)|\,dt\\
    &\le \int_{\log\epsilon}^{0}e^t(C_1\frac{\epsilon}{r}|t|+C_2\epsilon|t|)\,dt\\
    &\le\int_{\log\epsilon}^{0}e^t(C_1\frac{\epsilon}{r}|\log{\epsilon}|+C_2\epsilon|\log{\epsilon}|)\,dt\\
    &\le (C_1\frac{\epsilon}{r}|\log{\epsilon}|+C_2\epsilon|\log{\epsilon}|)(1-\epsilon)\\
    &\le C(1+\frac{1}{r})(\epsilon|\log{\epsilon}|).
\end{aligned}
\end{equation*}
Combining all these estimations together and keeping track with $r$, we have
$$|-p^{\epsilon}(r)-V(r)|\le C\epsilon|\log{\epsilon}|(1+\frac{1}{r}).$$\\
Then,
\[
\begin{aligned}
    |u^{\epsilon}(r)-u(r)|&=|\int_{1}^{r} p^{\epsilon}(s) \, ds-\int_1^{r}(-V(s))\,ds|\\
    &\le \int_{r}^{1}|p^{\epsilon}+V(s)|\,ds\\
    &\le C\epsilon |\log\epsilon| \int_{r}^{1}(1+\frac{1}{s})\,ds\\
    &\le C\epsilon|\log\epsilon|(1-r-\log r).
\end{aligned}
\]
\end{proof}
\begin{rem}
    The bound in Theorem 4.1 blows up as $r\rightarrow 0^+$. As of now, we do not know the convergence rate of $u^{\epsilon}(0)-u(0)$.
\end{rem}
\section{A discussion on the uniqueness of viscosity solutions}
In this section, we start with the following eikonal equation with Dirichlet boundary value condition 
\begin{equation}\label{eikonal2}
\left\{
\begin{aligned}
|Du| &= f(x) \quad \text{in } U,\\
u &= 0 \quad \text{on } \partial U,
\end{aligned} 
\right.    
\end{equation}\\
where $U$ is an open, bounded subset of $\mathbb{R}^n$ and $f: U \rightarrow [0,\infty) $ is a nonnegative continuous function.\\
We give the definitions of viscosity subsolutions, supersolutions, and solutions to \eqref{eikonal2}.
See \cite{Tran2021}.

\begin{defn}[viscosity solutions of \eqref{eikonal2}]
A function $u \in \BUC(U)$, where $\BUC(U)$ denotes the space of bounded, uniformly continuous functions on $U$, is called \\
(a) a viscosity subsolution of \eqref{eikonal2} if for any $\varphi\in C^1(U)$ such that $u-\varphi$ has a local max at $x_0\in U$, then 
\[
|D\varphi(x_0)|-f(x_0) \le 0, 
\]
and $u(x)\le0 $ on $\partial U$;\\
(b) a viscosity supersolution of \eqref{eikonal2} if for any $\psi \in C^1(U)$ such that $u-\psi$ has a local min at $x_0\in U$, then 
\[
|D\psi(x_0)|-f(x_0)\ge 0,
\]
and $u(x)\ge 0$ on $\partial U$;\\
(c) a viscosity solution of \eqref{eikonal2} if it is both a viscosity subsolution and a supersolution.
\end{defn}
Now we give the comparison principle for viscosity solutions of  \eqref{eikonal2} and the proof of it. This result is well known; see \cite{LeMitakeTran2017}. We include it here for completeness.
\begin{thm}
    Let $\mathcal{A}=\{x\in U| \;f(x)=0\}$ and $u$, $v$ be two viscosity solutions of \eqref{eikonal2}. Then, if $u\le v$ on $\mathcal{A}$, we have $u\le v$ in $U$.
\end{thm}
\begin{proof}
    Using the same notation as in the theorem. We suppose $u,v$ to be two viscosity solutions to \eqref{eikonal2} and $u\le v$ on $\mathcal{A}$.
    
    First, fix $\epsilon>0$. We consider the set $\mathcal{P}=\bigcup_{a\in\mathcal{A}}\cB(a,\delta)$, i.e., the union of balls with radius \( \delta \) centered at points in \( \mathcal{A} \). Since $u\le v$ in $\mathcal{A}$, we are able to choose $\delta_{\epsilon}>0$ small enough such that 
    \[
    su-(v+\epsilon)< -\frac{\epsilon}{2} \quad \text{in} \ \overline{\mathcal{P}}
    \]
for $s<1$ and $s$ close enough to 1.

For convenience, we denote $u^s = su$ and $v^{\epsilon}= v+\epsilon$ and our goal is to show $u^s \le v^{\epsilon}$ in $U$ for $\epsilon >0$ small and $s<1$ close to 1.\\
We continue our proof by using the classical ``doubling variable'' method. Assume by contradiction that 
\[
    \sup_{x\in \bar{U}\setminus\mathcal{P}}(u^s (x)-v^{\epsilon}(x)) = u^s (\bar x)-v^{\epsilon}(\bar x)= \sigma >0 \; \text{for some} \ \bar x\in \bar U\setminus \mathcal{P}.
\]
We see that:
\[
    u^{s}(x)-v^{\epsilon}(x) = su(x)-v(x)-\epsilon = 0-\epsilon<0 \; \text{on}\;\partial U \\
\]
and
\[
    u^s (x)-v^{\epsilon}(x) = su(x)-v(x)-\epsilon<-\frac{\epsilon}{2}<0 \;\text{on}\; \partial \mathcal{P}. 
\]\\
Hence, $\bar x $ can only belong to $U\setminus \mathcal{\bar P}$.

Here, we tried to avoid the case that $\bar x$ appears on the boundary because we will apply the viscosity subsolution and supersolution test later.
To be more precise, we need a local max or a local min near $\bar x$, which may not be true when $\bar x$ is on the boundary.
We consider an auxiliary function $\Psi(x,y) $:
\[
    \Psi^{\alpha}(x,y)=u^{s}(x)-v^{\epsilon}(y)-\frac{|x-y|^2}{2\alpha}
\]
and
\[
    \sup_{x,y\in \bar{U}\setminus\mathcal{P}}=\Psi ^{\alpha}(x_\alpha,y_\alpha) \; \text{for some} \; x_\alpha,y_\alpha \in \bar{U}\setminus\mathcal{P}.
\]
Since $(x_\alpha, y_\alpha) \rightarrow (\bar{x}, \bar{x})$ as $\alpha \rightarrow 0^+$, we have $x_\alpha, y_\alpha \in U \setminus \bar{\mathcal{P}}$ for sufficiently small $\alpha$.
Since $u$ is a viscosity solution to the \eqref{eikonal2}, we see that $u^s = su$ is a viscosity solution to the following equation:\\
\begin{equation}\label{eikonal3}
\left\{
\begin{aligned}
|Du^s| &= sf(x) \quad \text{in } U \\
u^s &= 0 \quad \text{on } \partial U
\end{aligned} 
\right.    
\end{equation}.\\
From the above, we know that the auxiliary function $\Psi^{\alpha}(x,y)$ has a max at $(x_\alpha,y_\alpha)$. Therefore, $x\mapsto \Psi^\alpha(x,y_\alpha)$ has a max at $x_\alpha$, which means
\[
    x\mapsto u^s (x)-\frac{|x-y_\alpha|^2}{2\alpha} \quad \text{has a max at} \; x_\alpha.
\]
As $u^s$ is a viscosity solution of \eqref{eikonal3}, by the viscosity subsolution test, we have 
\[
    \frac{1}{\alpha}|x_\alpha-y_\alpha|-sf(x_\alpha)\le0.
\]
Similarly, $y\mapsto \Psi^\alpha (x_\alpha,y)$ has a min at $y_\alpha$, which means 
\[
    y\mapsto v(y)-(-\frac{|x_\epsilon-y|^2}{2\alpha})\quad \text{has a min at} \; y_\alpha.
\]
As $v$ is a viscosity solution of \eqref{eikonal2}, by the viscosity supersolution test, we have 
\[
    \frac{1}{\alpha}|x_\alpha-y_\alpha|-f(y_\alpha)\ge0.
\]
Combine these two inequalities together, we get 
\[
    f(y_\alpha)\le \frac{1}{\alpha}|x_\alpha - y_\alpha|\le sf(x_\alpha).
\]
Take $\alpha \rightarrow 0^+$, we get 
\[
    f(\bar{x})\le sf(\bar{x}).
\]
Since $f(x)$ is nonnegative and  $\mathcal{A}=\{x\in U|f(x)=0\}$, then $f(x)>\gamma>0$ for some positive $\gamma$ in $U\setminus \bar{\mathcal{P}}$. Notice that  $\bar{x}\in U\setminus \bar{\mathcal{P}}$, $f(\bar{x})>0$. Recall that $s<1$, then the inequality 
\[
    f(\bar{x})\le sf(\bar{x})
\]
cannot be true. We get a contradiction. Hence we proved $u^s\le v^{\epsilon}$ in $U\setminus{\bar{\mathcal{P}}}$. 
Also, according to our construction at the beginning of the proof, we have $u^s\le v^\epsilon$ in $\mathcal{\bar{P}}$, too. As a result, $u^s\le v^\epsilon$ in the whole $U$. 

Moreover, $u^s\le v^\epsilon$ in $U$ is true for $s<1$ close to 1 and $\epsilon>0$ small. Then take $s\rightarrow 1^-$ and $\epsilon \rightarrow 0^+$, we have
\[
    u\le v \quad \text{in}\; U.
\]
\end{proof}
From this theorem, we can directly derive the following two corollaries.
\begin{cor}
    Let $u,v$ be two viscosity solutions of \eqref{eikonal2} and $u=v$ on $\mathcal{A}$. Then $u= v$ on $U$.
\end{cor}
\begin{cor}
For strictly positive $f$, if a viscosity solution to equation \eqref{eikonal2} exists, then it is unique.
\end{cor}
The proof of this corollary is quite simple. Notice that for strictly positive $f$, the set $\mathcal{A}$ is empty. Then for any two viscosity solutions $u,v$, the assumption that two solutions $u,v$ coincide on $\mathcal{A}$ holds vacuously, and hence the uniqueness follows from the extension property.  \\
Next, we provide an example to demonstrate the second corollary.
\begin{ex}
    Now, let us consider a simplified case where $U=(-1,1)$ and $f(x)=1, x\in(-1,1)$ which is strictly positive. Then \eqref{eikonal2} becomes
\begin{equation}\label{eikonal4}
\left\{
\begin{aligned}
|Du| &= 1 \quad \text{in } (-1,1), \\
u(-1)&=u(1) = 0.
\end{aligned} 
\right.    
\end{equation}
For equation \eqref{eikonal4}, we can show that it has a unique viscosity solution $u(x)$, where
\[
u(x) = 
\begin{cases}
-x+1, & \text{if } 0\le x < 1, \\
x+1,  & \text{if } -1<x<0.
\end{cases}
\]
\begin{figure}[h]
\centering
\begin{tikzpicture}

\begin{tikzpicture}
\begin{axis}[
    axis lines=middle,
    xlabel={$x$},
    ylabel={$u(x)$},
    ymin=-0.5, ymax=2,
    xmin=-3, xmax=3,
    samples=200,
    domain=-3:3,
    width=10cm, height=5cm,
    grid=none
]
\addplot[black, thick, domain=-1:0] {x + 1};
\addplot[black, thick, domain=0:1] {-x + 1};
\end{axis}
\end{tikzpicture}

\end{tikzpicture}
\caption{Plot of $u(x)$.}
\end{figure}
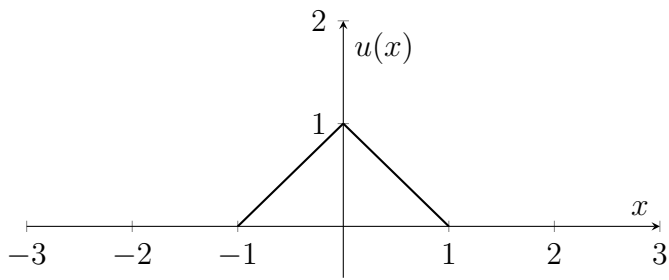
\end{ex}
When $f(x)$ is only nonnegative, things get more complicated. We are allowed to get different viscosity solutions. The following is an example when $f$ is just nonnegative.
See also \cite{LeMitakeTran2017}.
\begin{ex}
    Given a piecewise function $f$, where
\[
f(x) =
\begin{cases}
-x-\frac{1}{2}, & \text{if } -1<x\le-\frac{1}{2},\\
x+\frac{1}{2}, & \text{if } -\frac{1}{2}<x\le0,\\
-x+\frac{1}{2}, & \text{if } 0<x \le \frac{1}{2},\\
x-\frac{1}{2}, & \text{if } \frac{1}{2}< x<1.
\end{cases}
\]
\begin{figure}[h]
\centering
\begin{tikzpicture}
\begin{axis}[
    axis lines=middle,
    xmin=-1.5, xmax=1.5,
    ymin=-0.3, ymax=1,
    xtick={-1,-0.5,0,0.5,1},
    ytick={0,0.5},
    xlabel={$x$},
    ylabel={$f(x)$},
    samples=200,
    domain=-1:1,
    thick,
    grid=none,
    width=10cm, height=5cm,
    legend style={draw=none, at={(0.8,0.85)}}
]

\addplot[black, domain=-0.999:-0.5] {-x - 0.5};

\addplot[black, domain=-0.499:0] {x + 0.5};

\addplot[black, domain=0.001:0.5] {-x + 0.5};

\addplot[black, domain=0.501:0.999] {x - 0.5};

\end{axis}
\end{tikzpicture}
\caption{Plot of $f(x)$.}
\end{figure}
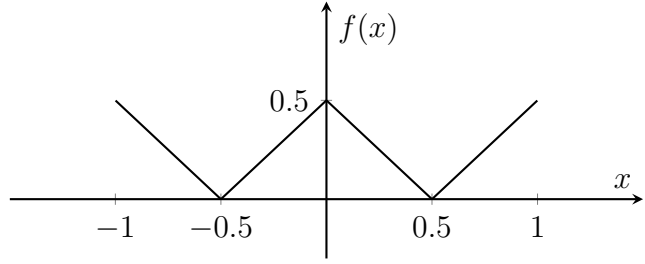
Here we still suppose $U=(-1,1)$. Then \eqref{eikonal2} becomes
\begin{equation}\label{eikonal5}
\left\{
\begin{aligned}
|Du| &= f(x) \quad \text{in } (-1,1), \\
u(-1)&=u(1) = 0.
\end{aligned} 
\right.    
\end{equation}
Equation \eqref{eikonal5} has different viscosity solutions. Consider $u_1(x)$ and $u_2(x)$,
\[
u_1(x)=
\begin{cases}
    \frac{1}{2}(x+\frac{1}{2})^2-\frac{1}{8}, & \text{if } -1<x\le0,\\
    \frac{1}{2}(x-\frac{1}{2})^2-\frac{1}{8}, & \text{if } 0<x<1.
\end{cases}
\]
\begin{figure}[h]
\centering
\begin{tikzpicture}
\begin{axis}[
    axis lines=middle,
    xlabel={$x$},
    ylabel={$u_1(x)$},
    domain=-1:1,
    samples=200,
    ymin=-0.5, ymax=0.5,
    xmin=-2, xmax=2,
    xtick={-1, 0, 1},
    ytick={-0.25, 0, 0.25},
    grid=none,
    width=10cm, height=5cm,  
    thick
]
\addplot[domain=-1:0, black, thick]
    {0.5 * (x + 0.5)^2 - 1/8};

\addplot[domain=0:1, black, thick]
    {0.5 * (x - 0.5)^2 - 1/8};

\end{axis}
\end{tikzpicture}
\caption{Plot of $u_1(x)$.}
\end{figure}
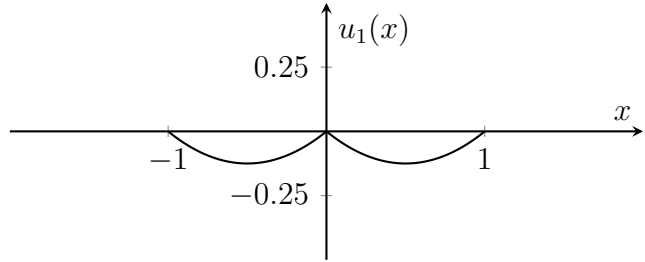
\[
u_2(x)=
\begin{cases}
    -\frac{1}{2}(x+\frac{1}{2})^2+\frac{1}{8}, & \text{if } -1<x\le-\frac{1}{2},\\
    \frac{1}{2}(x+\frac{1}{2})^2+\frac{1}{8}, & \text{if } -\frac{1}{2}<x\le0,\\
    \frac{1}{2}(x-\frac{1}{2})^2+\frac{1}{8}, & \text{if } 0<x\le \frac{1}{2},\\
    -\frac{1}{2}(x-\frac{1}{2})^2+\frac{1}{8}, & \text{if } \frac{1}{2}<x<1.
 \end{cases}
\]
\begin{figure}[h]
\centering
\begin{tikzpicture}
  \begin{axis}[
    axis lines=middle,
    xmin=-1.1, xmax=1.1,
    ymin=-0.1, ymax=0.4,
    xtick={-1, -0.5, 0, 0.5, 1},
    ytick={-0.1, 0, 0.125, 0.25, 0.375},
    xlabel={$x$},
    ylabel={$u_2(x)$},
    ticks=both,
    width=10cm,
    height=5cm,
    domain=-1:1,
    samples=400,
    enlargelimits=true,
    axis line style={->},
    ticklabel style={font=\small},
    every axis x label/.style={at={(current axis.right of origin)}, anchor=west},
    every axis y label/.style={at={(current axis.above origin)}, anchor=south},
  ]

  \addplot[domain=-1:-0.5, thick, black] {-(x + 0.5)^2/2 + 1/8};
  \addplot[domain=-0.5:0, thick, black] {(x + 0.5)^2/2 + 1/8};
  \addplot[domain=0:0.5, thick, black] {(x - 0.5)^2/2 + 1/8};
  \addplot[domain=0.5:1, thick, black] {-(x - 0.5)^2/2 + 1/8};

  \end{axis}
\end{tikzpicture}
\caption{Plot of $u_2(x)$.}
\end{figure}
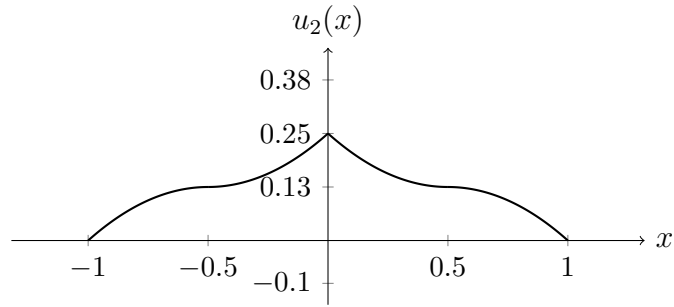
We can verify that $u_1$ and $u_2$ are both viscosity solutions to equation \eqref{eikonal5}.\\
This example demonstrates that when $f(x)$ is merely nonnegative, the uniqueness of the viscosity solution may fail. For our problem, we indeed have 
\[
\lim_{\epsilon \rightarrow 0^+ }u^{\epsilon}=u=u_2.
\]
We note that $u=u_2$ is indeed the maximal solution among all possible radially symmetric solutions.
\end{ex}
\section*{Acknowledgment}
This work constitutes my undergraduate thesis at the University of Wisconsin–Madison. I am deeply grateful to Professor Hung Vinh Tran for his invaluable guidance and support throughout the past year. His insight, encouragement, and patient supervision have been essential to the completion of this project.

I would also like to thank Professors Hiroyoshi Mitake and Xiaoqin Guo for their helpful comments and suggestions during discussions in Beijing, which contributed to the refinement of this work.

\end{document}